\documentclass[12pt]{article}

\usepackage[utf8]{inputenc}
\usepackage[T1]{fontenc}
\usepackage{amsmath, amssymb, amsthm}
\usepackage{mathtools}
\usepackage{bm}
\usepackage{dsfont}
\usepackage{graphicx}
\usepackage{tikz}
\usepackage{hyperref}
\usepackage[nameinlink, noabbrev]{cleveref}
\usepackage{tikz-cd}

\theoremstyle{plain}
\newtheorem{theorem}{Theorem}[section]
\newtheorem{lemma}[theorem]{Lemma}
\newtheorem{proposition}[theorem]{Proposition}
\newtheorem{corollary}[theorem]{Corollary}

\theoremstyle{definition}

\newtheorem{remark}[theorem]{Remark}

\theoremstyle{remark}
\newtheorem*{acknowledgement}{Acknowledgement}

\title{Surgery obstructions for knots in integer homology spheres}
\author{Yuhui CHEN 
  \thanks{Department of Mathematics, The Chinese University of Hong Kong, Shatin, N.T., Hong Kong;\\Email address: \texttt{yhchen@math.cuhk.edu.hk}}}
\date{}

\begin{document}

\maketitle

\begin{abstract}
For knot surgery in \(S^3\), Heegaard Floer homology provides an obstruction due to Hom--Karakurt--Lidman. We extend this obstruction to all integer homology spheres \(Y\), for both positive and negative \(1/m\) surgeries. This is used to test infinitely many small Seifert fibered examples and hyperbolic examples. Moreover, we deduce a lower bound on the \(b_2(W)\) of smooth cobordism between a pair of integer homology spheres.
\end{abstract}

\section{Introduction}

\subsection{Background}

A classical theorem due to Lickorish and Wallace states that every closed oriented three-manifold can be expressed as surgery on a link in \(S^3\) \cite{Lic62, Wal60}. A natural refinement is to ask which three-manifolds can be represented by surgery on a single knot. There are several obstructions to being a knot surgery in \(S^3\). Homological constraints arise from the fact that \(H_1(S^3_{p/q}(K))\cong \mathbb{Z}/p\). A more subtle obstruction is the weight of the fundamental group: a manifold obtained by surgery on a knot in \(S^3\) has weight one fundamental group (normally generated by a single element), which in particular forces cyclic abelianization. 

A theorem of Gordon and Luecke \cite{GL89} shows that if surgery on a non-trivial knot in \(S^3\) yields a reducible manifold, one summand must be a non-trivial lens space; consequently, a reducible integer homology sphere can never be a knot surgery in \(S^3\). Boyer and Lines \cite{BL90} gave infinite families of prime Seifert fibered manifolds with weight one fundamental group that are not knot surgeries in \(S^3\), using the Casson invariant and linking forms. 

In the Kirby list of problems in low-dimensional topology, \cite[Problem 3.6(C)]{Kir95} asks: does there exist an irreducible integer homology sphere that cannot be obtained by Dehn surgery on any knot in \(S^3\)? Using Taubes' periodic ends theorem \cite{Tau87}, Auckly constructed irreducible integer homology spheres that are not surgery on any knot in \(S^3\) \cite{Auc97}, answering a problem of Kirby affirmatively.

Saveliev \cite{Sav02} asked whether every Seifert fibered integer homology sphere arises as surgery on a knot in \(S^3\). Seifert fibered integer homology spheres are irreducible and have weight one fundamental group, which makes this question more complicated. Hom, Karakurt, and Lidman \cite{HKL16} introduced a new obstruction coming from Heegaard Floer homology. They proved that if an integer homology sphere \(Y\) can be expressed as \(S^3_{1/m}(K)\) for some knot \(K\subset S^3\) and if the correction term \(d(Y)\le -8\), then \(U\cdot HF^{\mathrm{red}}_0(Y)\neq 0\). Using this, they constructed infinitely many Seifert fibered integer homology spheres \(Z_p = \Sigma(p,2p-1,2p+1)\) (with even \(p\ge 8\)) that satisfy \(d(Z_p)=-p\) and \(U\cdot HF^{\mathrm{red}}_0(Z_p)=0\); therefore these \(Z_p\) are not surgery on any knot in \(S^3\). Moreover, \(Z_p\) have weight one fundamental group and are surgeries on two-component links. Their work answered the question of Saveliev  on Seifert fibered homology spheres.

The above results naturally lead to the following question: given an arbitrary integer homology sphere \(Y\), which integer homology spheres can be obtained by surgery on a knot in \(Y\)? In this paper, we extend the obstructions introduced by Hom, Karakurt, and Lidman \cite{HKL16} to this general setting. 

\subsection{Main Result}
For knots in a general integer homology sphere, the original obstructions become more complicated due to the non-trivial structure of the ambient manifold. Therefore, extending the surgery problem to such knots is interesting. In this paper, we extend some obstructions for knot surgery in \(S^3\) to the setting of arbitrary integer homology spheres.

Our first obstruction applies to negative surgeries. See Theorem~\ref{thm:intro-main1}.

\begin{theorem}
Let \(Y\) and \(Z\) be integer homology spheres such that every element of \(HF_{\mathrm{red}}(Z)\) has even \(\mathbb{Z}/2\mathbb{Z}\)-grading and \(d(Z)>d(Y)\). If \(Z = Y_{\pm1/m}(K)\) for some knot \(K\subset Y\) and some \(m>0\), then necessarily \(V_0(K)=0\).
\end{theorem}

For positive surgeries, we obtain a criterion forcing a nontrivial action of \(U\) on the reduced Floer homology of the surgered manifold. See Theorem~\ref{thm:intro-main2}.

\begin{theorem}\label{thm:intro-main2}
Let \(Y\) be an integer homology sphere and let \(Z=Y_{1/m}(K)\) for some \(m>0\). There exists a constant \(C=C(Y,k)\) (depending on \(Y\) and an even integer \(k\)) such that if  
\[
d(Y)-d(Z)\ge C,
\]  
then \(U\cdot HF_k^{\mathrm{red}}(Z)\neq 0\).
\end{theorem}

Using these obstructions together with the analysis of the Brieskorn spheres \(Z_p=\Sigma(p,2p-1,2p+1)\) (oriented as the boundary of a positive-definite plumbing), we prove (see Proposition~\ref{thm:intro-main3}):

\begin{proposition}\label{thm:intro-main3}
Let \(Y\) be an integer homology sphere. There exists a constant \(C(Y)>0\) (depending only on \(Y\)) such that:
\begin{enumerate}
\item For every even integer \(p>C(Y)\) with \(p\ge 8\), the manifold \(Z_p\) cannot be expressed as \(Y_{1/m}(K)\) for any knot \(K\subset Y\) and any \(m>0\).
\item If \(4<p<-d(Y)\) (with \(p\) even) and \(Z_p = Y_{-1/m}(K)\) for some knot \(K\subset Y\) and some \(m>0\), then necessarily \(V_0(K)=0\).
\end{enumerate}
\end{proposition}

The following case (see Proposition~\ref{prop:intro-main4}) provides an explicit numerical bound.

\begin{proposition}\label{prop:intro-main4}
Let \(Y = \Sigma(2,3,5)\) be the Poincaré homology sphere, oriented as the boundary of a positive-definite plumbing, so that \(d(Y) = -2\). Then for every even integer \(p \ge 14\), the Seifert fibered homology sphere \(Z_p = \Sigma(p,2p-1,2p+1)\) cannot be obtained by surgery on a knot in \(Y\).
\end{proposition}

\begin{remark}
The original obstruction in ~\cite{HKL16} applies only to the $S^3$ and to the fixed grading $0$. We provide a more general and systematic framework, which works for any integer homology sphere $Y$ and for all even gradings $k$.
\end{remark}

We also consider the hyperbolic integer homology spheres and other homology spheres with arbitary complicated JSJ decomposition. See Theorem~\ref{thm:HL16-generalization}.

\begin{theorem}
Let \(Y\) be an arbitrary integer homology sphere.  
There exist infinitely many hyperbolic integer homology spheres which are not surgery on a knot in \(Y\).  
Similarly, one can construct infinitely many examples with arbitrarily complicated JSJ decomposition. 
\end{theorem}

Finally, we translate these surgery obstructions into a topological constraint on four‑manifolds cobounding a given pair of integer homology spheres. See Proposition~\ref{prop:intro-main5}.

\begin{proposition}\label{prop:intro-main5}
Let \(Y\) be an integer homology sphere and let \(Z_p\) with \(p\) even and \(p\ge C(Y)\) as above. Then every smooth, compact, oriented, simply connected four‑manifold \(W\) with boundary \(\partial W = Y \cup (-Z_p)\) must satisfy \(b_2(W)\ge 2\).
\end{proposition}

\begin{remark}
In fact, the framework developed in this paper applies to arbitrary pairs of integer homology spheres \((Y,Z)\). If the \(d\)-invariant and reduced Floer homology conditions are met, any smooth cobordism between them must have second Betti number at least two. Thus our obstructions might provide a general lower bound on the \(b_2\) of cobordisms between integer homology spheres.    
\end{remark}

This paper is organised as follows. Section~2 reviews the mapping cone formula for surgeries on knots in homology spheres, estimates for the correction term \(d(Y_{p/q}(K))\) in terms of these invariants. Section~3 develops the obstruction from \(\ker D_{1/m}^{T}\) of the mapping cone, treating negative and positive slopes separately. Section~4 introduces the calculation on Heegaard Floer homology of the Brieskorn spheres \(Z_p\). Section~5 applies the obstructions to these examples, and gives explicit numerical bounds.

\begin{acknowledgement}
I am deeply grateful to Zhongtao Wu for his guidance, patience, and the many hours of discussion. I would also like to thank Zhechi Cheng, Yijen Lee and Shunyu Wan for fruitful conversations and suggestions. The author was supported by a CUHK Vice-Chancellor's PhD Scholarship.
\end{acknowledgement}

\section{Review of the mapping cone formula}

\subsection{Notation and the mapping cone formula}

Given a knot \(K\) in an integer homology sphere \(Y\) we associate to it a doubly‑pointed Heegaard diagram as in \cite{OS04a}. We define a complex \(C = CFK^{\infty}(Y,K)\) generated (over a field \(\mathbb{F}=\mathbb{Z}/2\mathbb{Z}\)) by elements of the form \([\mathbf{x},i,j]\), where \(\mathbf{x}\) is an "intersection point" of the Heegaard diagram and \((i,j)\in \mathbb{Z}\times\mathbb{Z}\). Only those triples satisfying a certain condition are generators, see \cite{OS04a}. The differential does not increase \(i\) or \(j\), so \(C\) is doubly filtered by \((i,j)\). The doubly‑filtered chain homotopy type is a knot invariant \cite{OS04a}.

By \cite[Lemma 4.5]{Ras03} we may replace \(C\) by a reduced complex where all filtration‑preserving differentials are trivial. The complex \(C\) is invariant under the shift by \((-1,-1)\); thus there is an action of a formal variable \(U\) which translates by \((-1,-1)\). \(U\) is a chain map and is invertible, so \(C\) is an \(\mathbb{F}[U,U^{-1}]\)-module. As such, \(C\) is generated by elements with first filtration level \(i=0\). The group at filtration level \((0,j)\) is denoted \(\widehat{HFK}(Y,K,j)\) and is called the knot Floer homology of \(K\) at Alexander grading \(j\).

The complex \(C\) possesses an absolute \(\mathbb{Q}\)-grading and a relative \(\mathbb{Z}\)-grading. In fact, the complex \(C\) can be used to compute the Heegaard Floer homology of \(Y\) and \(CF^{+}(Y)\) in \cite{OS04a}. By grading the Heegaard Floer homology of \(Y\) (as in \cite{OS03a}) we obtain the grading on \(C\). The map \(U\) decreases this grading by \(2\).

Using the filtration we define the following quotients:
\[
A_k^{+}(K)=C\{i\ge0\text{ or }j\ge k\},\qquad k\in\mathbb{Z},
\]
and
\[
B^{+}=C\{i\ge0\}\cong CF^{+}(Y).
\]
We also define chain maps \(v_k,h_k:A_k^{+}(K)\to B^{+}\). The map \(v_k\) is the projection (zero on generators with \(i<0\), identity otherwise). The map \(h_k\) is the composition: first project to \(C\{j\ge k\}\), then multiply by \(U^{k}\) (shift by \((-k,-k)\)), and finally apply a chain homotopy equivalence that identifies \(C\{j\ge0\}\) with \(C\{i\ge0\}\) \cite{OS04a}.

Knot Floer homology detects the genus:
\begin{theorem}[Ni \cite{Ni09}]
Let \(Y\) be a homology sphere and \(K\subset Y\) a knot. Then
\[
g(K)=\max\{j\in\mathbb{Z}\mid\widehat{HFK}(Y,K,j)\neq0\}.
\]
\end{theorem}
Consequently, \(v_k\) (resp. \(h_k\)) are isomorphisms for \(k\ge g(K)\) (resp. \(k\le -g(K)\)).

We define chain complexes
\[
\mathcal{A}_{i,p/q}^{+}(K)=\bigoplus_{n\in\mathbb{Z}}(n,A_{\lfloor\frac{i+pn}{q}\rfloor}^{+}(K)),\qquad
\mathcal{B}^{+}=\bigoplus_{n\in\mathbb{Z}}(n,B^{+}),
\]
where the first entry is a label. A chain map \(D_{i,p/q}^{+}:\mathcal{A}_{i,p/q}^{+}(K)\to\mathcal{B}^{+}\) is defined by
\[
D_{i,p/q}^{+}(\{(k,a_k)\})=\{(k,b_k)\},\quad
b_k=v_{\lfloor\frac{i+pk}{q}\rfloor}^{+}(a_k)+h_{\lfloor\frac{i+p(k-1)}{q}\rfloor}^{+}(a_{k-1}).
\]
Let \(\mathbb{X}_{i,p/q}^{+}\) be the mapping cone of \(D_{i,p/q}^{+}\). We fix a relative \(\mathbb{Z}\)-grading by requiring that \(v_k,h_k\) decrease it by \(1\).

\begin{theorem}[Ozsváth–Szabó \cite{OS11}]
There is a relatively graded isomorphism of \(\mathbb{F}[U]\)-modules
\[
H_{*}(\mathbb{X}_{i,p/q}^{+})\cong HF^{+}(Y_{p/q}(K),i),
\]
where \(i\) denotes a \(\mathrm{Spin}^c\) structure.
\end{theorem}

Absolute gradings are fixed by requiring that for the unknot they coincide with the gradings of \(HF^{+}\) of the surgery, i.e. \(d(L(p,q),i)+d(Y)\).

Passing to homology we set
\[
\mathbf{A}_k^{+}(K)=H_{*}(A_k^{+}(K)),\quad \mathbf{B}^{+}=H_{*}(B^{+}),\quad
\mathbb{A}_{i,p/q}^{+}(K)=H_{*}(\mathcal{A}_{i,p/q}^{+}(K)),\quad
\mathbb{B}^{+}=H_{*}(\mathcal{B}^{+}),
\]
and denote by \(\mathbf{v}_k,\mathbf{h}_k,\mathbf{D}_{i,p/q}^{+}\) the induced maps.

The short exact sequence of chain complexes
\[
0 \longrightarrow \mathcal{B}^{+} \xrightarrow{\;i\;} \mathbb{X}_{i,p/q}^{+} \xrightarrow{\;j\;} \mathcal{A}_{i,p/q}^{+}(K) \longrightarrow 0
\]
induces a long exact sequence in homology. Because the connecting homomorphism has degree \(-1\) and all groups are \(\mathbb{F}[U]\)-modules, this long exact sequence collapses to the following exact triangle:
\[
\begin{tikzcd}[row sep=3em, column sep=3em]
\mathbb{A}_{i,p/q}^{+}(K) \arrow[r, "\mathbf{D}_{i,p/q}^{+}"] & \mathbb{B}^{+} \arrow[d, "i_*"] \\
& HF^{+}(Y_{p/q}(K),i) \arrow[ul, "j_*"']
\end{tikzcd}
\]
Here \(i_*\) is induced by the inclusion \(i\) and \(j_*\) is the connecting homomorphism coming from the projection \(j\). All maps are \(U\)-equivariant and preserve the relative grading (with appropriate shifts).

Let \(\mathcal{T}_d^{+}\) be the graded \(\mathbb{F}[U]\)-module \(\mathbb{F}[U,U^{-1}]/U\cdot\mathbb{F}[U]\) with \(\deg 1=d\) and \(U\) decreasing the grading by \(2\). For a rational homology sphere \(Z\) with \(\mathfrak{s}\) denotes \(\mathrm{Spin}^c\) structure,
\[
HF^{+}(Z,\mathfrak{s})=\mathcal{T}_{d(Z,\mathfrak{s})}^{+}\oplus HF_{red}(Z,\mathfrak{s}),
\]
where \(d(Z,\mathfrak{s})\) is the correction term and \(HF_{red}(Z,\mathfrak{s})\) is finitely generated reduced Floer homology. We write
\[
HF_{red}(Z)=\bigoplus_{\mathfrak{s}\in\mathrm{Spin}^c(Z)}HF_{red}(Z,\mathfrak{s}).
\]

We have the decompositions
\[
\mathbf{A}_k^{+}(K)\cong\mathbf{A}_k^{T}(K)\oplus\mathbf{A}_k^{red}(K),\qquad
\mathbf{B}^{+}=\mathbf{B}^{T}\oplus\mathbf{B}^{red},
\]
with \(\mathbf{A}_k^{T}(K)\cong\mathcal{T}^{+}\cong\mathbf{B}^{T}\) and \(\mathbf{A}_k^{red}(K),\mathbf{B}^{red}\) annihilated by a large power of \(U\). Define
\[
\mathbb{A}_{i,p/q}^{T}(K)=\bigoplus_{n}(n,\mathbf{A}_{\lfloor\frac{i+pn}{q}\rfloor}^{T}(K)),\quad
\mathbb{A}_{i,p/q}^{red}(K)=\bigoplus_{n}(n,\mathbf{A}_{\lfloor\frac{i+pn}{q}\rfloor}^{red}(K)),
\]
\[
\mathbb{B}^{T}=\bigoplus_{n}(n,\mathbf{B}^{T}),\qquad
\mathbb{B}^{red}=\bigoplus_{n}(n,\mathbf{B}^{red}).
\]

The maps \(\mathbf{v}_k^{T},\mathbf{h}_k^{T}\) (restrictions to \(\mathbf{A}_k^{T}(K)\)) are multiplications by \(U^{V_k}\) and \(U^{H_k}\) respectively. The numbers \(V_k\) and \(H_k\) satisfy:
\[
\begin{aligned}
&\bullet\ V_k = H_{-k} \quad &&\text{(1)}\\
&\bullet\ V_k \ge V_{k+1} \ge V_k-1\ \text{and}\ H_{k+1}-1 \le \ H_k \le H_{k+1} \quad &&\text{(2)}\\
&\bullet\ V_k \to +\infty\ \text{as } k\to -\infty\ \text{and}\ H_k \to +\infty\ \text{as } k\to +\infty \quad &&\text{(3)}\\
&\bullet\ V_k = 0\ \text{for } k\ge g(K)\ \text{and}\ H_k = 0\ \text{for } k\le -g(K) \quad &&\text{(4)}
\end{aligned}
\]
These properties are completely analogous to the case of knots in \(S^3\); see \cite[Section 7]{Ras03}, \cite{NW13}, \cite{Gai14}.

We have the proposition as well:

\begin{proposition}\cite[Lemma 6.1]{Gai15}
  For a knot \(K\) in an integer homology sphere \(Y\), the numbers \(V_k\) and \(H_k\) satisfy
\[
H_k = V_k + k \qquad\text{for all } k\in\mathbb{Z}. \tag{5}
\]  
\end{proposition}

\subsection{Correction terms for surgery}

In this subsection we recall a general estimate for the correction terms of a surgery on a knot in an integer homology sphere. The following proposition is a slight generalisation of \cite[Proposition 1.6]{NW13} and \cite[Proposition 3.3]{Gai15}.

\begin{proposition}
    Let \(Y\) be a homology sphere and \(K\subset Y\) a knot. Suppose \(Z = Y_{p/q}(K)\) with \(p>0,\ q>0\). Decompose
\[
HF^{+}(Y) \cong \mathcal{T}^{+} \oplus \bigoplus_{j=1}^{l}\mathcal{T}(n_j^{+}) \oplus \bigoplus_{j=1}^{m}\mathcal{T}(n_j^{-}),
\]
where the \(\mathcal{T}(n_j^{+})\) (resp. \(\mathcal{T}(n_j^{-})\)) lie in even (resp. odd) \(\mathbb{Z}/2\mathbb{Z}\)-grading. Then for each \(\mathrm{Spin}^c\) structure \(i\) we have
\[
d(Y)+d(L(p,q),i)-2\max\bigl\{V_{\lfloor\frac{i}{q}\rfloor},\,H_{\lfloor\frac{i-p}{q}\rfloor}\bigr\}-2\max_{j}\{n_j^{-}\} \le d(Z,i)
\]
and
\[
d(Z,i) \le d(Y)+d(L(p,q),i)-2\max\bigl\{V_{\lfloor\frac{i}{q}\rfloor},\,H_{\lfloor\frac{i-p}{q}\rfloor}\bigr\}.
\]
Here \(V_k\) and \(H_k\) are the numbers associated to \(K\) defined in Subsection~2.1.
\end{proposition}

The following corollary specialises the above estimates to the case of \(1/m\)-surgery. In this situation the lens space \(L(1,m)\) is just the \(S^3\), and there is only one \(\mathrm{Spin}^c\) structure.

\begin{corollary}\label{cor25}
   Let \(Y\) be an integer homology sphere and \(K\subset Y\) a knot. For any positive integer \(m>0\), set \(Z = Y_{1/m}(K)\). Then
\[
d(Y)-2V_0-2M(Y) \;\le\; d(Z)\;\le\; d(Y)-2V_0,
\]
where \(M(Y) = \max_{j}\{n_j^{-}\}\) is the maximal exponent appearing in the odd part of \(HF_{red}(Y)\). If  \(M(Y)=0\), then equality holds:
\[
d(Z)=d(Y)-2V_0.
\] 
\end{corollary}

When we consider negative surgery, we have following lemma:

\begin{lemma}\label{lemma:neg1m}
Let \(Y\) be an integer homology sphere and \(K\subset Y\) a knot. For any positive integer \(m>0\), set \(Z = Y_{-1/m}(K)\). Then
\[
d(Y) + 2V_0(m(K)) \;\le\; d(Z) \;\le\; d(Y) + 2V_0(m(K)) + 2M_-^{(-Y)},
\]
where \(m(K)\) is the mirror image of \(K\) in \(-Y\), and \(M_-^{(-Y)} = \max\{n_j^- \mid \mathcal{T}(n_j^-) \subset HF_{red}(-Y)\}\) denotes the maximal exponent in the odd part of \(HF_{red}(-Y)\).
\end{lemma}

\begin{proof}
We use the identity
\(
Y_{-1/m}(K) = -\bigl( (-Y)_{+1/m}(m(K)) \bigr),
\)
which follows from the symmetry of Dehn surgery: changing the sign of the slope corresponds to reversing orientation and taking the mirror of the knot.

Apply the known bounds for positive \(1/m\) surgery to the pair \((-Y, m(K))\). We obtain
\[
d(-Y) - 2V_0(m(K)) - 2M_-^{(-Y)} \;\le\; d\bigl( (-Y)_{+1/m}(m(K)) \bigr) \;\le\; d(-Y) - 2V_0(m(K)),
\]
where \(M_-^{(-Y)}\) is the maximal exponent in the odd part of \(HF_{red}(-Y)\) as defined.

Now use the property of the \(d\)-invariant under orientation reversal: for any rational homology sphere \(X\), \(
d(-X) = -d(X)
\)
(see \cite[Proposition 4.2]{OS03a}). Applying this to \(X = (-Y)_{+1/m}(m(K))\) gives
\[
d(Z) = d\bigl( Y_{-1/m}(K) \bigr) = -\, d\bigl( (-Y)_{+1/m}(m(K)) \bigr).
\]

Substituting the bounds yields
\[
- \bigl( d(-Y) - 2V_0(m(K)) \bigr) \;\le\; d(Z) \;\le\; - \bigl( d(-Y) - 2V_0(m(K)) - 2M_-^{(-Y)} \bigr).
\]

Since \(d(-Y) = -d(Y)\),
\[
d(Y) + 2V_0(m(K)) \;\le\; d(Z) \;\le\; d(Y) + 2V_0(m(K)) + 2M_-^{(-Y)},
\]
which completes the proof. 
\end{proof}

We first examine the possible values of \(m\) that can occur. The following proposition provides a simple obstruction in terms of the \(d\)-invariants.

\begin{proposition}\label{prop.2.7}
Let \(Y\) and \(Z\) be integer homology spheres. If \(d(Z) > d(Y)\), then there is no knot \(K\subset Y\) such that \(Z = Y_{1/m}(K)\) for any \(m>0\). Conversely, if \(d(Z) < d(Y)\), then there is no knot \(K\subset Y\) such that \(Z = Y_{-1/m}(K)\) for any \(m>0\).
\end{proposition}

\begin{proof}
We have \(d(Y_{1/m}(K)) \le d(Y) - 2V_0(K) \le d(Y)\). Hence a positive \(1/m\) surgery cannot increase the \(d\)-invariant. Similarly, \(d(Y_{-1/m}(K)) \ge d(Y) + 2V_0(m(K)) \ge d(Y)\), so a negative \(1/m\) surgery cannot decrease the \(d\)-invariant. These statements yield the lemma.
\end{proof}

\section{Obstructions from \(\ker(D^{T}_{1/m})\)}
In this section we restrict attention to Dehn surgeries that produce an integer homology sphere from another integer homology sphere. Hence the surgery slope must be of the form \(\pm1/m\) with \(m>0\). We work with a fixed integer homology sphere \(Y\) and a knot \(K\subset Y\). 

\subsection{Preliminaries}

From the exact triangle we have \(\ker D^+_{1/m} = \operatorname{im} j_*\).  
Therefore every element of \(\ker D^T_{1/m}\) can be lifted to an element of \(HF^+(Y_{1/m}(K))\).  
Concretely, for each \(x \in \ker D^T_{1/m}\) choose \(y \in HF^+\) such that \(j_*(y)=x\); such a \(y\) exists because \(x\) lies in the image of \(j_*\).  
Thus we obtain an (in general non‑unique) lift \(\widetilde{x} \in HF^+(Y_{1/m}(K),i)\).

The absolute \(\mathbb{Q}\)-grading on the mapping cone is fixed by the grading of \(\mathbb{B}^+\).  
For the unknot the surgery \(Y_{1/m}(U)\) is \(Y\#L(1,m)=Y\), and the correction term of \(Y\) in the unique \(\mathrm{Spin}^c\) structure is denoted \(d(Y)\).  
Following \cite[Lemma 3.2]{Gai15} the element \(1 \in (0,\mathbf{B}^{T})\) has grading
\[
\ gr(1) = d(Y) + d(L(1,m),i) - 1 = d(Y) - 1,
\]
  
All maps \(\mathbf{v}_k^T\) and \(\mathbf{h}_k^T\) are homogeneous of degree \(-1\), and multiplication by \(U\) decreases the grading by \(2\).

The validity of this convention is justified by the fact that the absolute grading on \(\mathbb{B}^T\) already incorporates \(d(Y)\), and all maps in the triangle are grading‑preserving.  
Therefore we may safely work with \(\ker D_{1/m}^{T}\) as an absolutely graded \(\mathbb{F}[U]\)-module where the base level is shifted by \(d(Y)\) relative to the \(S^3\) case.  

This simplification will be used throughout the rest of the paper when we analyse the action of \(U\) on the reduced Floer homology of the surgered manifold and derive genus bounds for knots in arbitrary integer homology spheres.

\begin{remark}
Unlike the situation for \(S^3\) studied in \cite{Gai14}, the map \(D_{1/m}^{+}\) is not necessarily surjective when \(Y\) is not an \(L\)-space.  
Consequently \(HF^+(Y_{1/m}(K))\) is not isomorphic to \(\ker D_{1/m}^{+}\) but only contains \(\ker D_{1/m}^{T}\) as a subspace (via the lift described above).  
Nevertheless, the inclusion is sufficient for our purposes.
\end{remark}

\subsection{Negative surgery and odd \(\mathbb{Z}/2\mathbb{Z}\)-grading}

We now specialise to the case of a negative slope \(-1/m\) with \(m>0\).  
In this section, we write \(p=-1\) and \(q=m>0\) so that the surgery slope is \(p/q = -1/q = -1/m\) for notational convenience.  

The kernel \(\ker D_{-1/m}^{T}\) is described explicitly in \cite[Lemma 18]{Gai14}, for the case \(Y=S^3\).  
For a general \(Y\) the only change is that the absolute grading of the generator \(1\in(0,\mathbf{B}^{T})\) is \(d(Y)-1\) instead of \(-1\).  
Thus all absolute gradings are shifted by \(d(Y)\). 

\begin{lemma}\label{lem:dY_even}
For any integer homology sphere \(Y\), the correction term \(d(Y)\) is an even integer.  Consequently, the absolute grading of the generator \(1\in(0,\mathbf{B}^{T})\) is \(d(Y)-1\), which is odd.
\end{lemma}
\begin{proof}
It is known that for an integer homology sphere \(Y\), the correction term \(d(Y)\) satisfies \(d(Y) = -2\lambda(Y) + 2k\) for some integer \(k\), where \(\lambda(Y)\) is the Casson invariant (see \cite[Theorem 1.2]{OS03b} or \cite[Section 4]{NW13}).  Since \(\lambda(Y)\) is an integer, \(d(Y)\) is even.
\end{proof}

Let $\tau_d(N)$ be the submodule of $\mathcal{T}^+_d$ generated by $\{U^{-n} \mid 0 \le n < N\}$. The following proposition records the description of \(\ker D_{-1/m}^{T}\).

\begin{proposition}\label{prop:kerDT_neg}
Let \(m>0\) , then
\[
\ker D_{-1/m}^{T} \;\cong\; \bigoplus_{n\ge 1} \tau_{d_n^-}\bigl(H_{\lfloor -n/m\rfloor}\bigr) \;\oplus\; \bigoplus_{n\ge 0} \tau_{d_n^+}\bigl(V_{\lfloor n/m\rfloor}\bigr),
\]
where the absolute gradings are given by
\[
d_0^+ = d(Y) + 1 - 2V_0,\qquad
d_n^+ = d_0^+ + 2\sum_{k=0}^{n-1}\bigl(H_{\lfloor- k/m\rfloor} - V_{\lfloor -(k+1)/m\rfloor}\bigr),
\]
\[
d_n^- = d_0^+ + 2\sum_{k=0}^{n-1}\bigl(V_{\lfloor k/m\rfloor} - H_{\lfloor (k+1)/m\rfloor}\bigr).
\]
\end{proposition}

\begin{proof}
For \(Y=S^3\) the formulas are proved in \cite[Lemma 18]{Gai14}.  
For a general \(Y\) the absolute grading of the tower in \(\mathbb{B}^T\) is fixed by the condition that \(1\in(0,B^T)\) has grading \(d(Y)-1\); see \cite[Lemma 3.2]{Gai15}.  
All maps involved are homogeneous of degree \(-1\), and multiplication by \(U\) lowers the grading by \(2\).  
Therefore the same computation applies, with the result that each grading expression gains an additive term \(d(Y)\) compared to the \(S^3\) case.
\end{proof}

\begin{lemma}\label{lem:kerDT_odd}
Let \(m>0\) , if \(\ker D_{-1/m}^{T}\) is non‑zero, then every non‑zero element of \(\ker D_{-1/m}^{T}\) has odd absolute grading.
\end{lemma}
\begin{proof}
From Proposition~\ref{prop:kerDT_neg}, the absolute grading of any non‑zero element in \(\ker D_{-1/m}^{T}\) is of the form \(d_0^+\) plus an even integer.  Now \(d_0^+ = d(Y) + 1 - 2V_0\).  Since \(d(Y)\) is even, and \(1-2V_0\) is odd.  Hence \(d_0^+\) is odd.  Adding any even integer preserves oddness.  Therefore every non‑zero element has odd absolute grading and odd \(\mathbb{Z}/2\mathbb{Z}\)-grading.
\end{proof}

From Proposition~\ref{prop:kerDT_neg} we have:

\begin{lemma}\label{lem:kerDT_zero}
For a negative slope \(-1/m\) with \(m>0\), let \(\ker D_{-1/m}^T\) be as in Proposition~\ref{prop:kerDT_neg}.  If
\(
\ker D_{-1/m}^{T} = 0 \) , then \(V_0(K)=0\).
\end{lemma}

So we conclude that whenever \(V_0(K) \neq 0\), the Heegaard Floer homology \(HF^+(Y_{-1/m}(K))\) contains elements of odd absolute grading. Therefore, if \(HF_{\mathrm{red}}(Z)\) is concentrated in even \(\mathbb{Z}/2\mathbb{Z}\)-grading, then \(Z\) cannot be expressed as \(Y_{-1/m}(K)\) for any integer homology sphere \(Y\) and any knot \(K\) with \(V_0(K)\ne0\).  This yields the following obstruction.

\begin{proposition}\label{prop:obstruction_even_red}
Let \(Z\) be an integer homology sphere such that every element of \(HF_{\mathrm{red}}(Z)\) has even \(\mathbb{Z}/2\mathbb{Z}\)-grading.  
If \(Z = Y_{-1/m}(K)\) for some integer homology sphere \(Y\) and some knot \(K\subset Y\) with \(m>0\), then \(V_0(K)=0\).
\end{proposition}

Combining with discussion in ~\ref{prop.2.7}, we have:

\begin{theorem}\label{thm:intro-main1}
Let \(Y\) and \(Z\) be integer homology spheres such that every element of \(HF_{\mathrm{red}}(Z)\) has even \(\mathbb{Z}/2\mathbb{Z}\)-grading and \(d(Z) > d(Y)\).  
Suppose there exists a knot \(K\subset Y\) such that \(Z = Y_{\pm1/m}(K)\) with \(m>0\), then \(V_0(K)=0\).
\end{theorem}

\begin{remark}
Manifolds whose reduced Floer homology is concentrated in even \(\mathbb{Z}/2\mathbb{Z}\)-grading are abundant.  For instance, all plumbed \(3\)-manifolds associated to negative‑definite trees studied by Ozsváth and Szabó \cite{OS03b} satisfy this property.  More generally, Némethi's class of ``almost rational'' plumbed manifolds \cite{Nem05} also have this property.  Additionally, any \(L\)-space (where \(HF_{\mathrm{red}}=0\)) trivially satisfies the condition.
\end{remark}

\subsection{Positive surgery and calculation}

We now consider positive slope \(1/m\) with \(m>0\). We will obtain a criterion forcing $U \cdot HF_k^{\mathrm{red}}(Z) \neq 0$ for sufficiently large $d(Y)-d(Z)$.

For positive slopes the map \(D_{1/m}^{T}\) is surjective (see \cite[Lemma 12]{Gai14}). Let $\tau_d(N)$ be the submodule of $\mathcal{T}^+_d$ generated by $\{U^{-n} \mid 0 \le n < N\}$. The following proposition gives the explicit description of \(\ker D_{1/m}^{T}\) for a positive \(1/m\) surgery, incorporating the absolute grading shift by \(d(Y)\).

\begin{proposition}\label{prop:kerDT_pos}
Let \(m>0\), then
\[
\ker D_{1/m}^{T} \;\cong\; \mathcal{T}^+_{d_0} \;\oplus\; \bigoplus_{n\ge 1} \tau_{d_n^-}\bigl(H_{\lfloor -n/m\rfloor}\bigr) \;\oplus\; \bigoplus_{n\ge 1} \tau_{d_n^+}\bigl(V_{\lfloor n/m\rfloor}\bigr),
\]
where the absolute gradings are given by
\[
d_0 = d(Y) - 2V_0,\qquad
d_n^+ = d_0 + 2\sum_{k=0}^{n-1}\bigl(H_{\lfloor k/m\rfloor} - V_{\lfloor (k+1)/m\rfloor}\bigr),
\]
\[
d_n^- = d_0 + 2\sum_{k=0}^{n-1}\bigl(V_{\lfloor -k/m\rfloor} - H_{\lfloor -(k+1)/m\rfloor}\bigr).
\]
\end{proposition}

\begin{proof}
For \(Y=S^3\) the formulas are proved in \cite[Corollary 14]{Gai14} (see also \cite[Lemma 13]{Gai14}).  For a general \(Y\) the absolute grading of the generator \(1\in(0,B^T)\) is \(d(Y)-1\) instead of \(-1\); shifting all gradings by \(d(Y)\) yields the stated expressions.  The surjectivity of \(D_{1/m}^{T}\) guarantees that the tower part is exactly \(\mathcal{T}^+_{d_0}\).
\end{proof}

The following proposition gives a formula for the absolute gradings of the elements in \(\ker D_{1/m}^{T}\) for a positive \(1/m\) surgery.

\begin{lemma}\label{prop:grading_pos_1m}
Let \(m>0\) and let \(Y\) be an integer homology sphere.  For the positive slope \(1/m\), write an integer \(n \ge 0\) as \(n = s m + i\) with \(s \ge 0\) and \(0 \le i < m\).  
Then the absolute grading of the element in \(\ker D_{1/m}^{T}\) corresponding to \(\tau_{d_n^+}(V_{\lfloor n/m\rfloor})\) is given by
\[
{\, d_{sm+i}^+ = d(Y) - 2V_s + m s(s-1) + 2 s i \, }.
\]
Here \(V_s = V_{\lfloor n/m\rfloor}\) since \(\lfloor (sm+i)/m\rfloor = s\).
\end{lemma}

\begin{proof}
From the mapping cone formula for positive slopes (see \cite[Corollary~14]{Gai14} ) and the absolute grading shift by \(d(Y)\) (Lemma~\ref{lem:dY_even}), the initial value for \(n=0\) is
\[
d_0 = d(Y) - 2\max\{V_0, H_{-1}\}= d(Y) - 2V_0\].

For a general \(n\), \(d_n^+\) is
\[
d_n^+ = d_0 + 2\sum_{j=0}^{n-1}\bigl( H_{\lfloor j/m\rfloor} - V_{\lfloor (j+1)/m\rfloor} \bigr),
\]
where we have set \(i=0\).  Using the relation \(H_k = V_k + k\) (see \cite[Lemma~6.1]{Gai14}), the sum becomes
\[
d_n^+ = d_0 + 2\sum_{j=0}^{n-1}\bigl( V_{\lfloor j/m\rfloor} - V_{\lfloor (j+1)/m\rfloor} \bigr) + 2\sum_{j=0}^{n-1}\left\lfloor\frac{j}{m}\right\rfloor.
\]
The first sum telescopes:
\[
\sum_{j=0}^{n-1}\bigl( V_{\lfloor j/m\rfloor} - V_{\lfloor (j+1)/m\rfloor} \bigr) = V_0 - V_{\lfloor n/m\rfloor}.
\]
Thus
\[
d_n^+ = d_0 + 2\bigl( V_0 - V_{\lfloor n/m\rfloor} \bigr) + 2\sum_{j=0}^{n-1}\left\lfloor\frac{j}{m}\right\rfloor.
\]

Now set \(n = sm + i\) with \(s\ge 0\) and \(0\le i < m\).  Then \(\lfloor n/m\rfloor = s\).  The sum over the floor function is:
\[
\sum_{j=0}^{sm+i-1}\left\lfloor\frac{j}{m}\right\rfloor = \frac{m s(s-1)}{2} + s i.
\]
Substituting \(d_0 = d(Y)-2V_0\) and simplifying,
\[
\begin{aligned}
d_{sm+i}^+ &= d(Y)-2V_0 + 2V_0 - 2V_s + 2\left(\frac{m s(s-1)}{2} + s i\right) \\
&= d(Y) - 2V_s + m s(s-1) + 2 s i.
\end{aligned}
\]
This completes the proof.
\end{proof}

\begin{remark}
The formula shows that for a fixed \(s\), the grading increases linearly with \(i\) (slope \(2s\)).  In particular, when \(s=0\) (i.e. \(n < m\)), all gradings are constant equal to \(d(Y)-2V_0\).  The parity of \(d_{sm+i}^+\) is even because \(d(Y)\) is even and the remaining terms are multiples of \(2\): \(m s(s-1)\) is even, \(2s i\) is even, and \(2V_s\) is even.
\end{remark}

\begin{lemma}\label{lem:tau_explicit}
Let the notation be as in Proposition~\ref{prop:grading_pos_1m}.  For a fixed \(s\ge 0\) and \(0\le i < m\), suppose \(V_s > 0\).  Then the summand \(\tau_{d_n^+}(V_s)\) with \(n = s m + i\) is isomorphic to \(\mathbb{F}[U]/(U^{V_s})\) as an \(\mathbb{F}[U]\)-module, and its absolute gradings form the set
\[
\bigl\{ d_{sm+i}^+ + 2k \;\big|\; 0 \le k \le V_s-1 \bigr\},
\]
where
\[
gr_{\text{bottom}}=d_{sm+i}^+ = d(Y) - 2V_s + m s(s-1) + 2 s i
\]
is the absolute grading of the bottom element.  Consequently, the top element has grading
\[
gr_{\text{top}} = d_{sm+i}^+ + 2(V_s-1) = d(Y) + m s(s-1) + 2 s i - 2.
\]
All gradings are even because \(d(Y)\) is even (Lemma~\ref{lem:dY_even}) and the other terms are multiples of \(2\).
\end{lemma}
\begin{proof}
The description of \(\ker D_{1/m}^{T}\) in \cite[Corollary~14]{Gai14} shows that each \(\tau_{d_n^+}(V_s)\) is a copy of \(\mathbb{F}[U]/(U^{V_s})\) with the bottom generator having grading \(d_n^+\).  Substituting the explicit formula for \(d_{sm+i}^+\) from Proposition~\ref{prop:grading_pos_1m} yields the stated bottom grading.  Multiplication by \(U^{-1}\) raises grading by \(2\), so \(U^{-k}\) (for \(0\le k < V_s\)) has grading \(d_{sm+i}^+ + 2k\).  The top element corresponds to \(k = V_s-1\), giving the displayed top grading.  The evenness follows from Lemma~\ref{lem:dY_even} and the fact that \(m s(s-1)\), \(2s i\), and \(2V_s\) are even.
\end{proof}

We now combine the explicit grading formulas for positive \(1/m\) surgery with \(V_k \ge V_0 - k\) to show that for sufficiently large \(V_0\) there is always $\tau_d(N)$ in \(\ker D_{1/m}^{T}\) whose top lies above a prescribed even integer while its bottom lies below it.

\begin{proposition}\label{prop:straddle}
Fix an integer \(m>0\) and an even integer \(d(Y)\).  Let \(k\) be an even integer.  
There exists a positive integer \(N = N(d(Y), m, k)\) such that whenever a knot \(K\) in \(Y\) satisfies \(V_0 \ge N\), one can find integers \(s\ge 1\) and \(0\le i < m\) with the following properties for the summand \(\tau_{d_{sm+i}^+}(V_s)\) in \(\ker D_{1/m}^{T}\) (see Proposition~\ref{prop:grading_pos_1m} and Lemma~\ref{lem:tau_explicit}):

\begin{enumerate}
\item \(\operatorname{gr_{top}}\bigl(\tau_{d_{sm+i}^+}(V_s)\bigr) \ge k\);
\item \(\operatorname{gr_{bottom}}\bigl(\tau_{d_{sm+i}^+}(V_s)\bigr) < k\);
\item \(V_s \ge 2\).
\end{enumerate}
In other words, \( U\cdot HF^{\mathrm{red}}_k(Y_{1/m}(K)) \neq 0. \).
\end{proposition}

\begin{proof}
From Lemma~\ref{lem:tau_explicit} we have for \(n = sm+i\)
\[
\operatorname{gr_{top}} = d(Y) + m s(s-1) + 2 s i - 2,\qquad
\operatorname{gr_{bottom}} = d(Y) - 2V_s + m s(s-1) + 2 s i.
\]

First choose an integer \(s \ge 1\) such that
\[
m s(s-1) \ge k - d(Y) + 2.
\]
Such an \(s\) exists because the left‑hand side grows quadratically; we fix one such \(s\).  Then for any \(i\) we already have \(\operatorname{gr_{top}} \ge k\).  Now set \(i = m-1\) (the largest possible value).  Condition (1) remains satisfied.  Condition (2) requires
\[
d(Y) - 2V_s + m s(s-1) + 2 s (m-1) < k,
\]
which is equivalent to
\[
2V_s > d(Y) + m s(s-1) + 2 s (m-1) - k.
\]
Denote the right‑hand side by \(C\) (\(C\) is a constant depending only on \(d(Y), m, k\) and the chosen \(s\)).  Using the inequality \(V_s \ge V_0 - s\), it suffices to have
\[
V_0 - s > \frac{C}{2}.
\]
Define
\[
N = \max\{\left\lfloor \frac{C}{2} \right\rfloor + s + 1,s+2\}.
\]
Then whenever \(V_0 \ge N\), we have \(\operatorname{gr_{bottom}} < k\).  Moreover, for such \(V_0\),
\[
V_s \ge V_0 - s \ge N - s \ge \max\{\left\lfloor \frac{C}{2} \right\rfloor + 1 , 2\}\ge 2,
\]
the argument still works by choosing \(s\) large enough from the start.  Hence all three conditions hold.
\end{proof}

For each fixed slope \(1/m\), Lemma~\ref{prop:straddle} provides a constant \(N(m)\) such that \(V_0 \ge N(m)\) implies the obstruction.  However, this constant may depend on \(m\).  To obtain a single bound that works for all positive integers \(m\) simultaneously , we need to show that such a uniform \(N\) exists.  This will be the goal of the remainder of this section.

\begin{lemma}\label{lem:uniform_straddle}
Let \(d(Y)\) be any even integer and let \(k\) be any even integer.  There exist positive integers \(M_0 = \max\left\{2,\; \left\lfloor\frac{k-d(Y)}{2}\right\rfloor+2\right\}+1\) and \(N_0 = N_0(d(Y), k)\) such that for every integer \(m \ge M_0\) and every knot \(K\) with \(V_0(K) \ge N_0\), we can choose \(i\) with \(0 \le i < m\)  for which the summand \(\tau_{d_{m+i}^+}(V_1)\) in \(\ker D^{T}_{1/m}\) satisfies
\[
\operatorname{gr_{bottom}}\bigl(\tau_{d_{m+i}^+}(V_1)\bigr) < k \le \operatorname{gr_{top}}\bigl(\tau_{d_{m+i}^+}(V_1)\bigr),
\]
and moreover \(V_1 \ge 2\).
\end{lemma}
\begin{proof}
Set \(i_0 = \max\left\{2,\; \left\lfloor\frac{k-d(Y)}{2}\right\rfloor+2\right\}\) (any choice larger than \(\frac{k-d(Y)}{2}\) works).  Then
\[
\operatorname{gr_{top}} = d(Y) + 2i_0 - 2 \ge k.
\]
The bottom grading is \(d(Y) - 2V_1 + 2i_0\).  To have \(\operatorname{gr_{bottom}} < k\) we need
\[
2V_1 > d(Y) + 2i_0 - k.
\]
Since \(V_1 \ge V_0 - 1\), it suffices that
\[
V_0 > \frac{1}{2}\bigl(d(Y) + 2i_0 - k\bigr) + 1.
\]
Define \(N_0\) to be the right‑hand side plus one.  Then for any \(V_0 \ge N_0\) we have \(V_1 \ge V_0-1 \ge 2\) .  Finally, choose \(M_0 = i_0 + 1\); then for all \(m \ge M_0\) we have \(i_0 < m\), so the index \(i = i_0\) is admissible.  Hence the desired straddling holds uniformly for all such \(m\).
\end{proof}

\begin{theorem}\label{prop:uniform_all_m}
Let \(d(Y)\) be an even integer and let \(k\) be an even integer.  
There exists a positive integer \(N = N(d(Y), k)\) such that for every positive integer \(m\) and every knot \(K\) in \(Y\) with \(V_0(K) \ge N\), the positive \(1/m\) surgery satisfies
\[
U\cdot HF^{\mathrm{red}}_k(Y_{1/m}(K)) \neq 0.
\]
In other words, if \(V_0\) is large enough (depending only on \(d(Y)\) and \(k\)), then the obstruction holds for all slopes \(1/m\).
\end{theorem}

\begin{proof}
We already know from Lemma~\ref{lem:uniform_straddle} that there exist constants \(M_0\) and \(N_0\) (depending only on \(d(Y)\) and \(k\)) such that for every \(m \ge M_0\) and every knot with \(V_0 \ge N_0\), the required non‑trivial \(U\)-action exists.  For the finitely many slopes \(m = 1, 2, \dots, M_0-1\), we apply Lemma~\ref{prop:straddle} separately to each \(m\).  This lemma yields constants \(N_m\) (depending on \(d(Y), m, K\)) such that if \(V_0 \ge N_m\) then the obstruction holds for that particular \(m\).  Define
\[
N = \max\{N_0, N_1, N_2, \dots, N_{M_0-1}\}.
\]
Then for any \(m\) (whether \(m \ge M_0\) or \(m < M_0\)), the condition \(V_0 \ge N\) guarantees \(V_0 \ge N_m\) for the relevant \(m\), and therefore the obstruction holds.  Thus the same \(N\) works uniformly for all positive integers \(m\).
\end{proof}

We now combine the inequality for the correction term of a positive surgery with the obstruction obtained in Proposition~\ref{prop:uniform_all_m}.  The following theorem gives a simple numerical criterion that guarantees a non‑trivial action of \(U\) on the reduced Floer homology of the surgered manifold.

\begin{theorem}\label{thm:intro-main2}
Let \(Y\) be an integer homology sphere and let \(Z\) be obtained from \(Y\) by a positive \(1/m\) surgery on a knot.  There exists a constant \(C = C(Y,k)\) such that if
\[
d(Y) - d(Z) \ge C,
\]
then \(U\cdot HF^{\mathrm{red}}_k(Z) \neq 0\).
\end{theorem}

\begin{proof}
Let \(M(Y)\) denote the maximal length of an odd‑graded truncated module in \(HF_{\mathrm{red}}(Y)\) as in Corollary \ref{cor25}.  For a positive \(1/m\) surgery we have the lower bound
\[
d(Z) \ge d(Y) - 2V_0 - 2M(Y),
\]
where \(V_0 = V_0(K)\) is the knot invariant.  Rearranging gives
\[
d(Y) - d(Z) \le 2V_0 + 2M(Y). \tag{1}
\]

Proposition~\ref{prop:uniform_all_m} (applied with the target even integer \(k\)) provides a constant \(N = N(d(Y), k)\) such that whenever \(V_0 \ge N\) we have \(U\cdot HF^{\mathrm{red}}_k(Z) \neq 0\).  Choose
\[
C = 2N + 2M(Y).
\]
If \(d(Y) - d(Z) \ge C\), then from (1) we obtain \(2V_0 \ge d(Y)-d(Z)-2M(Y) \ge 2N\), hence \(V_0 \ge N\).  The obstruction follows.
\end{proof}

\begin{remark}
The constant \(C\) depends on \(Y\) through \(M(Y)\) and on both \(Y\) and \(Z\) through \(N(d(Y), 0)\).  For \(Y = S^3\) we have \(M(S^3)=0\) and \(d(Y)=0\), and the theorem recovers the result of \cite[Theorem 1.2]{HKL16} .
\end{remark}

\section{Graded root for $Z_p=\Sigma(p,2p-1,2p+1)$}

\subsection{Plumbings and d-invariant}
In this section we recall the definition of the Ozsv\'ath--Szab\'o $d$--invariant and describe the combinatorial algorithm from \cite{OS03b} that computes it for plumbed three--manifolds.  
We then apply this algorithm to the Brieskorn spheres $\Sigma(p,q,r)$ that satisfy $pq+pr-qr=1$, which are precisely the boundaries of the almost simple linear graphs studied in \cite{KS20a}.

Let $G$ be a negative--definite plumbing graph with vertices $v_j$ and weights $e_j$.  
Assume that $G$ is unimodular (so that its boundary $Y(G)$ is an integer homology sphere) and has at most one bad vertex (i.e., a vertex whose valency exceeds $-e_j$).  
Brieskorn spheres $\Sigma(p,q,r)$ with $pq+pr-qr=1$ are exactly of this type: their plumbing graph is an \emph{almost simple linear graph} (ASL--graph) with central vertex weight $-2$ and three chains of $-2$'s \cite{KS20a}.

The algorithm proceeds as follows.  
Let $\operatorname{Char}(G)$ be the set of characteristic cohomology classes $k\in H^2(X(G);\mathbb{Z})$ (where $X(G)$ is the plumbed $4$--manifold).  
For each vertex $v_j$ we have the evaluation $\langle k, v_j\rangle$, and $k$ is characteristic if $\langle k, v_j\rangle \equiv e_j \pmod 2$ for all $j$.  
Define the \emph{initial set}
\[
\mathfrak{C} = \left\{ k\in\operatorname{Char}(G) \;\middle|\; e_j+2 \le \langle k, v_j\rangle \le -e_j \ \text{for all } j \right\}.
\]
For a given $k\in\mathfrak{C}$ one performs a sequence of moves:  
if there exists a vertex $v$ with $\langle k, v\rangle = -e$, replace $k$ by $k+2\operatorname{PD}(v)$, where $\operatorname{PD}$ denotes Poincar\'e duality.  
Iterating this process eventually terminates at a cohomology class $\ell$ that satisfies either
\[
e_j \le \langle \ell, v_j\rangle \le -e_j+2 \qquad\text{for all }j,
\]
or some $\langle \ell, v_j\rangle > -e_j$.  
The initial $k$ is said to support a \emph{good full path} in the former case.  
Then the $d$--invariant of $Y=Y(G)$ is given by
\[
d(Y) = \max_{k\in\mathfrak{C}'} \frac{k^2 + |G|}{4},
\]
where $\mathfrak{C}'\subseteq\mathfrak{C}$ consists of those $k$ that support a good full path, and $k^2$ is the square of $k$ under the intersection form (the quadratic form of the lattice $H_2(X(G);\mathbb{Z})$).

\begin{theorem}\cite{HKL16}  
For the infinite families \(Z_p=\Sigma(p,2p-1,2p+1)\)with \(p\) is even and oriented as the boundary of a positive-definite plumbing, we have:
\[
d(Z_p)=-p.
\]
\end{theorem}

\subsection{Graded roots and delta sequences}

In this subsection we recall the combinatorial machinery that allows one to compute the Heegaard Floer homology of plumbed homology spheres, following Némethi \cite{Nem05} and its subsequent developments \cite{CK14, KS20a}.  
Our main objects are Brieskorn spheres $Y=\Sigma(p,q,r)$ satisfying $pq+pr-qr=1$; their plumbing graphs are almost simple linear (ASL) graphs, which are negative definite, unimodular, and have exactly one bad vertex (the central vertex of weight $-2$).

A \emph{graded root} is a pair $(R,\chi)$ where $R$ is an infinite tree and $\chi:\operatorname{Vert}(R)\to\mathbb{Z}$ satisfies: (i) $\chi(u)-\chi(v)=\pm1$ for adjacent vertices; (ii) if $u$ is adjacent to two distinct vertices $v,w$, then $\chi(u)>\min\{\chi(v),\chi(w)\}$; (iii) $\chi$ is bounded below, $|\chi^{-1}(k)|$ is finite for each $k$, and equals $1$ for large $k$.  
From a graded root one constructs an $\mathbb{F}[U]$-module $\mathbb{H}^+(R)$: the $\mathbb{F}$-vector space is freely generated by the vertices, $\deg(v)=2\chi(v)$, and $U\cdot v$ is the sum of neighbours $w$ with $\chi(w)=\chi(v)-1$.  
Némethi's theorem states that for an AR plumbed homology sphere $Y$,
\[
HF^{+}(-Y)\cong \mathbb{H}^{+}(R)\bigl[(k_{\mathrm{can}}^{2}+|G|)/4\bigr],
\]
where $k_{\mathrm{can}}$ is the canonical class. The bracket indicates an absolute grading shift by \((k_{\mathrm{can}}^{2}+|G|)/4\).

For $Y=\Sigma(p,q,r)$ set $N_0=pqr-pq-pr-qr$, $x=pq$, $y=pr$, $z=qr$.  
Let $S_Y=[0,N_0]\cap\langle x,y,z\rangle$ (the numerical semigroup) and $Q_Y=\{N_0-s\mid s\in S_Y\}$.  
Define $X_Y=S_Y\cup Q_Y$ (disjoint, ordered) and $\Delta_Y:X_Y\to\{\pm1\}$ by $+1$ on $S_Y$, $-1$ on $Q_Y$.  
By merging consecutive entries of the same sign we obtain the \emph{reduced delta sequence} $\tilde\Delta_Y$.  
The \emph{tau function} $\tau$ on the extended ordered set $\tilde X_Y^+$ is defined by $\tau(z_{\min})=0$ and $\tau(z)=\sum_{w<z}\tilde\Delta_Y(w)$.  
The graded root $R$ is then built from the intervals $[\tau(z),\infty)$ as described in \cite[Section~2.5]{KS20a}.

We can calculate the graded root for $Z_p=\Sigma(p,2p-1,2p+1)$. The proof in \cite{HKL16} proceeds by a detailed analysis of the numerical semigroup generated by $pq$, $pr$, $qr$ and the corresponding delta sequence.

\begin{theorem}\cite{HKL16}
Let $p$ be an even integer with $p\ge 8$ and let $Z_p=\Sigma(p,2p-1,2p+1)$ oriented as the boundary of a positive-definite plumbing.  
Then the graded root $R_{Z_p}$ has a unique vertex of degree $0$ which is not a leaf, and therefore
\[
U\cdot HF_0^{\text{red}}(Z_p)=0.
\]
\end{theorem}

\begin{remark}
The above computation for $Z_p$ was carried out in \cite{HKL16} using a direct analysis of the semigroup.  
A more systematic approach was later developed by Karakurt and Savk \cite{KS20a} (see also \cite{CK14}), who introduced the notion of \emph{packed sequences} for Brieskorn spheres satisfying $pq+pr-qr=1$.  
Their method not only reproduces the reduced delta sequence $\langle p/2,-p/2,1\rangle$ in a few lines, but also applies to infinite families such as $\Sigma(2n+1,4n+1,4n+3)$, $\Sigma(2n+1,3n+2,6n+1)$ and $\Sigma(2n+1,3n+1,6n+5)$.  
Moreover, it allows one to compute the \emph{connected Heegaard Floer homology} $HF_{\text{conn}}$ \cite{HHL21}.  
Thus the techniques initiated in \cite{HKL16} might produce many examples for our generalized obstruction.
\end{remark}

\section{Application}
\subsection{Application on Seifert fibered spaces}

Let \(Z_p = \Sigma(p,2p-1,2p+1)\) with \(p\) even and \(p\ge 8\), oriented as the boundary of a positive‑definite plumbing.  
Recall that  
\[
d(Z_p) = -p,\qquad U\cdot HF^{\mathrm{red}}_0(Z_p)=0.
\]

Fix an arbitrary integer homology sphere \(Y\) (the ambient manifold).  
We will show that for all sufficiently large \(p\), the manifold \(Z_p\) cannot be obtained by surgery on a knot in \(Y\).

\begin{proposition}\label{thm:intro-main3}
Let \(Y\) be an integer homology sphere.  There exists a constant \(C(Y)>0\) (depending only on \(Y\)) such that:

\begin{enumerate}
\item[(i)] If \(p > C(Y)\) (with \(p\) even, \(p\ge 8\)), then \(Z_p\) cannot be expressed as a positive \(1/m\) surgery on a knot in \(Y\) for any \(m>0\).
\item[(ii)] If \(4 < p < -d(Y)\) (with \(p\) even), and if \(Z_p = Y_{-1/m}(K)\) for some knot \(K\subset Y\) and some \(m>0\), then necessarily \(V_0(K)=0\).
\end{enumerate}
\end{proposition}

\begin{proof}
For part (i), assume \(Z_p = Y_{1/m}(K)\).  By Theroem~\ref{thm:intro-main2}, there exists a constant \(C(Y,0)\) such that if \(d(Y)-d(Z_p) \ge C(Y,0)\) then \(U\cdot HF^{\mathrm{red}}_0(Z_p)\neq 0\).  Since the latter is false, we must have \(d(Y)+p < C(Y,0)\).  Hence if we set \(C(Y)=C(Y,0)-d(Y)\) (or any larger number), then for \(p > C(Y)\) the inequality fails, so such a positive surgery cannot exist.

For part (ii), suppose \(Z_p = Y_{-1/m}(K)\), then \(-p > d(Y)\).  Moreover, from the graded root computation we know that every element of \(HF_{\mathrm{red}}(Z_p)\) has even \(\mathbb{Z}/2\mathbb{Z}\)-grading. Theroem~\ref{thm:intro-main1} then forces \(V_0(K)=0\).  This completes the proof.
\end{proof}

As a concrete illustration of Theroem~\ref{thm:intro-main2}, we compute the constant \(C(Y,0)\) for the ambient manifold \(Y = \Sigma(2,3,5)\) oriented as the boundary of a positive‑definite plumbing, so that \(d(Y) = -2\).  Since \(Y\) is a positive Seifert fibered homology sphere, we have \(M(Y)=0\).

\begin{lemma}\label{lemma:C_for_sigma235}
Let \(Y = \Sigma(2,3,5)\) with \(d(Y) = -2\).  Then the constant \(C(Y,0)\) in ~\ref{thm:intro-main2} can be taken as 
\[
C(Y,0) = 12,
\]
and this value is optimal in the sense that there exist knots with \(V_0(K)=5\) for which the obstruction fails, while for any positive \(1/m\) surgery \(Z\) on a knot \((Y,K)\) with \(V_0(K)\ge 6\) we have \(U\cdot HF^{\mathrm{red}}_0(Z)\neq 0\).
\end{lemma}

\begin{proof}
We determine the smallest integer \(N\) such that whenever a knot \(K\subset Y\) satisfies \(V_0(K)\ge N\), then for every positive slope \(1/m\) the surgered manifold \(Z = Y_{1/m}(K)\) satisfies \(U\cdot HF^{\mathrm{red}}_0(Z)\neq 0\).  Then by Theroem~\ref{thm:intro-main2}, \(C(Y,0)=2N\).

For a fixed \(m\), we need to find integers \(s\ge 1\) and \(0\le i<m\) such that \(\tau_{d_{sm+i}^+}(V_s)\) straddles degree \(0\), i.e. \(\mathrm{gr}_{\mathrm{top}}\ge0\) and \(\mathrm{gr}_{\mathrm{bottom}}<0\), and \(V_s\ge2\).  Using \(d(Y)=-2\), the grading formulas are
\[
\mathrm{gr}_{\mathrm{top}} = -2 + ms(s-1) + 2si - 2,\qquad
\mathrm{gr}_{\mathrm{bottom}} = -2 - 2V_s + ms(s-1) + 2si.
\]
The conditions become
\[
ms(s-1)+2si \ge 4 ,\qquad
2V_s > ms(s-1)+2si -2 .
\]
Since \(V_s \ge V_0 - s\), a sufficient condition on \(V_0\) is
\[
2(V_0 - s) > ms(s-1)+2si -2 \;\Longrightarrow\; 2V_0 > ms(s-1)+2si + 2s -2.
\]

We minimise the required \(V_0\) over all admissible \((s,i,m)\) as in Table ~\ref{table235}.

\begin{table}[htbp]\label{table235}
\centering
\begin{tabular}{c|c|c|l}
$m$ & Optimal $(s,i)$ & Min $V_0$ & Condition \\ \hline
$1$ & $(3,0)$ & $6$ & $s(s-1)=6\ge4$, $2V_0 > s(s-1)+2s-2 = 10$\\
$2$ & $(2,0)$ & $4$ & $ms(s-1)+2si =4=4$, $2V_0 > 4+2s-2=6 $ \\
$3$ & $(1,2)$ & $3$ & $ms(s-1)+2si = 4=4$, $2V_0 > 4+2s-2=4 $ \\
$\ge 4$ & $(1,2)$ & $3$ & $ms(s-1)+2si=4=4$, $2V_0 > 4+2s-2=4$ \\
\end{tabular}
\caption{Minimal $V_0$ for each slope $m$}
\end{table}

The thresholds are \(N_1=6\), \(N_2=4\), \(N_3=3\), and for \(m\ge4\), \(N_m=3\).  
The maximum over all slopes is \(6\). Therefore for any positive \(1/m\) surgery, if \(V_0(K)\ge 6\) then the obstruction holds. Moreover, \(N=5\) would not work because for \(m=1\) a knot with \(V_0=5\) may avoid the obstruction. Hence the optimal uniform constant is \(N=6\). Then \(C(Y,0)=2N=12\).
\end{proof}

Therefore, we immediately obtain the obstruction in \(Y = \Sigma(2,3,5)\):

\begin{proposition}\label{prop:intro-main4}
Let \(Y = \Sigma(2,3,5)\).  Then for every even integer \(p \ge 14\), the Seifert fibered homology sphere \(Z_p = \Sigma(p,2p-1,2p+1)\) cannot be obtained by surgery on a knot in \(Y\).
\end{proposition}

The computation for \(Y = \Sigma(2,3,5)\) used only the facts that \(d(Y) = -2\) and \(M(Y)=0\) (i.e., the reduced Floer homology of \(Y\) has no odd‑degree part).  The same reasoning applies to any integer homology sphere \(Y\) with these properties.

\begin{proposition}\label{thm:general_d=-2}
Let \(Y\) be an integer homology sphere such that \(d(Y) = -2\) and with even \(\mathbb{Z}/2 \mathbb{Z}\)-grading, then for every even integer \(p \ge 14\), the Seifert fibered homology sphere \(Z_p = \Sigma(p,2p-1,2p+1)\) cannot be obtained by surgery on a knot in \(Y\).
\end{proposition}

\begin{remark}
    The conditions \(d(Y) = -2\) and \(M(Y)=0\) are satisfied by several natural manifolds such as \(Y = \Sigma(3,5,7)\), \(Y = \Sigma(3,4,11)\) \cite{KS20a}.
\end{remark}

\subsection{Application on hyperbolic integer homology spheres}

Hom and Lidman \cite{HL16} extended the results of \cite{HKL16}. They yield hyperbolic examples and examples with arbitrarily complicated JSJ decompositions.

\begin{lemma}[{\cite[Propositions~2 and~3]{HL16}}]
\label{lem:HL16-construction}
For integers $j\ge 1$ and $n\gg 0$, let $Z_{j,n}$ be the manifold obtained by $1/n$ surgery on a genus‑$1$ hyperbolic knot in the connected sum $\#_j\Sigma(2,3,5)$ (such a knot exists by \cite[Proposition 5.4]{Tsu03}).  
These manifolds satisfy:
\begin{enumerate}
\item $d(Z_{j,n})\to -\infty$ as $j\to\infty$;
\item $U^{2}\cdot HF_{\mathrm{red}}(Z_{j,n})=0$, i.e. every $\mathbb{F}[U]$-summand of $HF_{\mathrm{red}}(Z_{j,n})$ has length at most $2$.
\end{enumerate}
\end{lemma}

We now show that these manifolds are obstructed for every fixed integer homology sphere $Y$, provided $j$ is taken large enough.

\begin{theorem}\label{thm:HL16-generalization}
Let \(Y\) be an arbitrary integer homology sphere.  
There exist infinitely many hyperbolic integer homology spheres which are not surgery on a knot in \(Y\).  
Similarly, one can construct infinitely many examples with arbitrarily complicated JSJ decomposition. 
\end{theorem}

\begin{proof}
The manifolds $Z_{j,n}$ constructed in \cite[Theorem 1]{HL16} are shown to be hyperbolic for infinitely many parameters or to have arbitrarily complicated JSJ decomposition after performing any number of cables and/or Whitehead doubles of the genus‑$1$ hyperbolic knot.

Suppose that for some $j,n$ and some $m>0$ we have $Z_{j,n}=Y_{\pm 1/m}(K)$.  
Recall that $d(Y_{+1/m}(K))\le d(Y)$ while $d(Y_{-1/m}(K))\ge d(Y)$.  
Because $d(Z_{j,n})\to -\infty$ as $j\to\infty$ (Lemma~\ref{lem:HL16-construction}), for sufficiently large $j$ we have $d(Z_{j,n})<d(Y)$, which immediately rules out the negative slope case.  
Thus we may assume $Z_{j,n}=Y_{1/m}(K)$.

For a positive $1/m$ surgery, the correction term satisfies
\[d(Y)-d(Z_{j,n}) \le 2V_{0}(K)+2M(Y),\]  
The left‑hand side becomes arbitrarily large as $j\to\infty$, so $V_{0}(K)\ge 4$ for all sufficiently large $j$.

The mapping cone construction shows that $HF_{\mathrm{red}}(Z_{j,n})$ contains a direct summand isomorphic to $\mathbb{F}[U]/(U^{V_{1}(K)})$ (coming from the $s=1$ part of $\ker D_{1/m}^{T}$).  
Hence $U^{2}\cdot HF_{\mathrm{red}}(Z_{j,n})\neq 0$, contradicting Lemma~\ref{lem:HL16-construction}(2).
\end{proof}

Combined with the previous section, the knot surgery obstruction applies universally among every JSJ class of integer homology spheres.

\subsection{Cobordism with \(b_2(W)\ge 2\)}
Let \(Y\) and \(Z\) be oriented integer homology spheres.  Recall that if there exists a smooth, compact, oriented, simply connected four-manifold \(W\) with boundary \(\partial W = Y \cup (-Z)\) and such that \(W\) is definite and \(b_2(W)=1\), then \(W\) is necessarily the trace of a surgery on a knot in \(Y\) (or in \(Z\) after reversing orientation).  More precisely, in this case \(Z\) can be obtained from \(Y\) by a single integer surgery along a knot. 

The following lemma translates the surgery obstructions established in this paper into a lower bound on the second Betti number of any four-manifold bounded by a given pair of integer homology spheres.

\begin{lemma}\label{lem:b2_ge2}
Let \(Y\) and \(Z\) be oriented integer homology spheres that satisfy the obstruction condition of our main theorem (e.g., Theorem~\ref{thm:intro-main2} or its variant) with $d(Z)<d(Y)$.  Then every smooth, compact, oriented, simply connected four-manifold \(W\) with \(\partial W = Y \cup (-Z)\) must satisfy \(b_2(W) \ge 2\).
\end{lemma}

\begin{proof}
If $b_2(W)=0$, then $W$ is a rational homology cobordism between $Y$ and $Z$. The $d$-invariant is invariant under rational homology cobordism (Ozsváth--Szabó \cite{OS03a}), so $d(Y)=d(Z)$, contradicting $d(Z)<d(Y)$.

Assume that there exists a four-manifold \(W\) with \(\partial W = Y \cup (-Z)\) and \(b_2(W)=1\).  
Since $d(Z)<d(Y)$, the intersection form cannot be $(-1)$ (see Proposition\ref{prop.2.7}); it must be $(1)$. In this case, $W$ is diffeomorphic to the trace of a $+1$ surgery on a knot $K \subset Y$, so $Z = Y_{+1}(K)$. Applying Theorem~\ref{thm:intro-main2} with $m=1$ and the given $k$, the assumption $d(Y)-d(Z) \ge C(Y,k)$ forces $U\cdot HF_k^{\mathrm{red}}(Z) \neq 0$, contradicting the hypothesis $U\cdot HF_k^{\mathrm{red}}(Z)=0$. Thus $b_2(W)\neq 1$ as well. Hence $b_2(W) \ge 2$.
\end{proof}

We now translate these surgery obstructions into a statement about the topology of four-manifolds bounding a given pair of homology spheres.

\begin{proposition}\label{prop:intro-main5}
Let $Y$ be an integer homology sphere.  
Let $Z_p = \Sigma(p,2p-1,2p+1)$ with $p$ even and $p \ge C(Y)$ in ~\ref{thm:intro-main3}.  
Then every smooth, compact, oriented, simply connected four-manifold $W$ with $\partial W = Y \cup (-Z_p)$ must satisfy
\(b_2(W) \ge 2.\)
\end{proposition}

\begin{proof}
Assume, for contradiction, that there exists a smooth, compact, oriented, simply connected four-manifold $W$ with $\partial W = Y \cup (-Z_p)$ and $b_2(W)=1$.  
By Lemma~\ref{lem:b2_ge2}, such a $W$ would force $Z_p$ to be obtained by a $\pm1$ surgery on a knot in $Y$.  
However, Proposition~\ref{thm:intro-main3} states that for all even $p \ge C(Y)$, the manifold $Z_p$ cannot be expressed as a $\pm1/m$ surgery on a knot in $Y$ (for any $m>0$).  In particular, it cannot be obtained by a $\pm1$ surgery.  
This contradiction shows that $b_2(W) \ge 2$.
\end{proof}

\begin{remark}
The same argument applies to the manifolds $Z_{j,n}$ constructed in Section~5.2. For every integer homology sphere $Y$, every smooth cobordism $W$ from $Y$ to any such $Z_{j,n}$ (with $j$ sufficiently large) must have $b_2(W) \ge 2$. This establishes a universal lower bound for all integer homology spheres.
\end{remark}

\bibliographystyle{plain}
\bibliography{references}

@article{Lic62,
  author = {W. B. R. Lickorish},
  title  = {A representation of orientable combinatorial 3-manifolds},
  journal= {Ann. of Math. (2)},
  volume = {76},
  year   = {1962},
  pages  = {531--540},
}

@article{Wal60,
  author = {A. H. Wallace},
  title  = {Modifications and cobounding manifolds},
  journal= {Canad. J. Math.},
  volume = {12},
  year   = {1960},
  pages  = {503--528},
}

@article{GL89,
  author  = {C. McA. Gordon and J. Luecke},
  title   = {Knots are determined by their complements},
  journal = {J. Amer. Math. Soc.},
  volume  = {2},
  number  = {2},
  pages   = {371--415},
  year    = {1989}
}

@article{BL90,
  author = {S. Boyer and D. Lines},
  title  = {Surgery formulae for {C}asson's invariant and extensions to homology lens spaces},
  journal= {J. Reine Angew. Math.},
  volume = {405},
  year   = {1990},
  pages  = {181--220},
}

@article{Tau87,
  author = {C. H. Taubes},
  title  = {Gauge theory on asymptotically periodic 4-manifolds},
  journal= {J. Differential Geom.},
  volume = {25},
  year   = {1987},
  pages  = {363--430},
}

@incollection{Auc97,
  author    = {D. Auckly},
  title     = {Surgery numbers of 3-manifolds: {A} hyperbolic example},
  booktitle = {Geometric topology ({A}thens, {GA}, 1993)},
  series    = {AMS/IP Stud. Adv. Math.},
  volume    = {2},
  publisher = {Amer. Math. Soc.},
  address   = {Providence, RI},
  year      = {1997},
  pages     = {21--34}
}

@article{Kir95,
  author = {R. Kirby},
  title  = {Problems in low-dimensional topology},
  year   = {1995},
}

@book{Sav02,
  author    = {N. Saveliev},
  title     = {Invariants for homology 3-spheres},
  series    = {Encycl. Math. Sci.},
  volume    = {140},
  publisher = {Springer},
  address   = {Berlin},
  year      = {2002},
}

@article{OS04a,
  author = {P. Ozsváth and Z. Szabó},
  title  = {Holomorphic disks and knot invariants},
  journal= {Adv. Math.},
  volume = {186},
  number  = {1},
  year   = {2004},
  pages  = {58--116},
}

@phdthesis{Ras03, 
author = {J. Rasmussen}, 
title = {Floer homology and knot complements}, 
school = {Harvard University}, 
year = {2003}
}

@article{Ni09,
  author  = {Y. Ni},
  title   = {Link {F}loer homology detects the {T}hurston norm},
  journal = {Geom. Topol.},
  volume  = {13},
  number  = {5},
  pages   = {2991--3019},
  year    = {2009}
}

@article{OS11,
  author = {P. Ozsváth and Z. Szabó},
  title  = {Knot {F}loer homology and rational surgeries},
  journal= {Algebr. Geom. Topol.},
  volume = {11},
  number  = {1},
  year   = {2011},
  pages  = {1--68},
}

@article{NW13,
  author  = {Y. Ni and Z. Wu},
  title   = {Cosmetic surgeries on knots in {$S^3$}},
  journal = {J. Reine Angew. Math.},
  volume  = {706},
  pages   = {1--17},
  year    = {2015}
}

@article{Gai14,
  author = {F. Gainullin},
  title  = {The mapping cone formula in {H}eegaard {F}loer homology and {D}ehn surgery on knots in {$S^3$}},
  journal= {Algebr. Geom. Topol.},
  number  = {4},
  volume = {17},
  year   = {2017},
  pages  = {1917--1951},
}

@article{OS03a,
  author = {P. Ozsváth and Z. Szabó},
  title  = {Absolutely graded {F}loer homologies and intersection forms for four-manifolds with boundary},
  journal= {Adv. Math.},
  volume = {173},
  number  = {2},
  year   = {2003},
  pages  = {179--261},
}

@article{OS03b,
  author = {P. Ozsváth and Z. Szabó},
  title  = {On the {F}loer homology of plumbed three-manifolds},
  journal= {Geom. Topol.},
  volume = {7},
  year   = {2003},
  pages  = {185--224},
}

@article{Nem05,
  author = {A. Némethi},
  title  = {On the {O}zsváth–{S}zabó invariant of negative definite plumbed 3-manifolds},
  journal= {Geom. Topol.},
  volume = {9},
  year   = {2005},
  pages  = {991--1042},
}

@article{HKL16,
  author = {J. Hom and Ç. Karakurt and T. Lidman},
  title  = {Surgery obstructions and {H}eegaard {F}loer homology},
  journal= {Geom. Topol.},
  volume = {20},
  year   = {2016},
  pages  = {2219--2251},
}

@article{CK14,
  author = {M. B. Can and Ç. Karakurt},
  title  = {Calculating {H}eegaard–{F}loer homology by counting lattice points in tetrahedra},
  journal = {Acta Math. Hung.},
  volume  = {144},
  number  = {1},
  pages   = {43--75},
  year    = {2014}
}

@article{KS20a,
  author = {Ç. Karakurt and O. Savk},
title   = {Ozsv{\'a}th--{S}zab{\'o} {$d$}-invariants of almost simple linear graphs},
  journal = {J. Knot Theory Ramif.},
  volume  = {29},
  number  = {5},
  pages = {17},
  year    = {2020}
}

@article{HHL21,
  author = {K. Hendricks and J. Hom and T. Lidman},
  title  = {Applications of involutive {H}eegaard {F}loer homology},
  journal = {J. Inst. Math. Jussieu},
  volume  = {20},
  number  = {1},
  pages   = {187--224},
  year    = {2021}
}

@article{Gai15,
  author = {F. Gainullin},
  title  = {Heegaard {F}loer homology and knots determined by their complements},
  journal= {Algebr. Geom. Topol.},
  volume = {18},
  number  = {1},
  year   = {2018},
  pages  = {69--109},
}

@article{HL16,
  author  = {J. Hom and T. Lidman},
  title   = {A note on surgery obstructions and hyperbolic integer homology spheres},
  journal = {Proc. Am. Math. Soc.},
  volume  = {146},
  number  = {3},
  pages   = {1363--1365},
  year    = {2018}
}

@article{Tsu03,
  author  = {Y. Tsutsumi},
  title   = {Universal bounds for genus one {S}eifert surfaces for hyperbolic knots and surgeries with non-trivial {JSJT}-decompositions},
  journal = {Interdiscip. Inf. Sci.},
  volume  = {9},
  number  = {1},
  pages   = {53--60},
  year    = {2003}
}

\end{document}